\documentclass[12pt]{amsart}

\usepackage[english]{babel}
\usepackage[utf8]{inputenc}
\usepackage{amsmath}
\usepackage{amssymb}
\usepackage{graphicx}
\usepackage{xy}
\usepackage{hyperref}
\usepackage{url}
\usepackage[colorinlistoftodos]{todonotes}

\newtheorem{thm}{Theorem}
\newtheorem{lem}[thm]{Lemma}

\theoremstyle{remark}
\newtheorem{defn}[thm]{Definition}
\newtheorem{rem}[thm]{Remark}

\newtheorem*{notn}{Notation}
\numberwithin{thm}{section} \numberwithin{equation}{section}

\newcommand{\simon}[1]{\todo[color=green]{SR: #1}}
\newcommand{\urban}[1]{\todo[color=blue!40]{UL: #1}}
\newcommand{\mcG}{\mathcal{G}}
\newcommand{\mcN}{\mathcal{N}}
\newcommand{\mcP}{\mathcal{P}}

\DeclareMathOperator{\mex}{mex}

\title{Grundy values of Fibonacci nim}

\author{Urban Larsson}
\address{Department of Mathematics and Statistics, Dalhousie University, 6316 Coburg Road, PO Box 15000, Halifax, Nova Scotia, Canada B3H 4R2}
\email{urban031@gmail.com}

\author{Simon Rubinstein-Salzedo}
\address{Department of Statistics, 390 Serra Mall, Stanford University, Stanford, CA 94305, USA}
\email{simonr@stanford.edu}

\date{\today}

\begin{document}
\maketitle

\begin{abstract}
In this article, we investigate the Grundy values of the popular game of Fibonacci nim. The winning strategy, which amounts to understanding positions of Grundy value 0, was known since~\cite{Whinihan63}. In this paper, we extend Whinihan's analysis by computing all the positions of Grundy value at most 3. Furthermore, we show that, when we delete the Fibonacci numbers (which have Grundy value 0), the Grundy values of the starting positions are increasing, and we give upper and lower bounds on the growth rate.
\end{abstract}

\section{Introduction}
Fibonacci Nim, described and analyzed in~\cite{Whinihan63}, is a 2-player combinatorial game, popular due to its simple game rules and its elegant solution. Its analysis involves not only the Fibonacci numbers, but also the Zeckendorf representation of a natural number. It is played on one heap of tokens and the rules are the same for both players; thus the game is impartial (see~\cite{BCG01}). 

The rules of the game are as follows. Suppose that there are originally $n$ tokens in the heap. On the first move, the first player can remove between 1 and $n-1$ tokens. If, on the previous move, the last player removed $r$ tokens, then the next player can remove up to $2r$ tokens. The game ends when there are no moves left; the player left without a move loses.

Many impartial games are studied under the disjunctive sum operator; that is, two games $G$ and $H$ are played together, with a move in their sum $G+H$ being either a move in $G$ or a move in $H$, but not both. Sums of games are highly amenable to analysis, due to the Sprague-Grundy theory \cite{Sprague35,Grundy39}, which we review in \S\ref{sec:sprague-grundy}. Fibonacci Nim, however, is a so-called move-size dynamic game, where the current player's move options depend on the particular removal by the previous player, and so the possible moves of the game depend not only on the position but also on the game history. There are two logical ways of summing games of Fibonacci nim, or equivalently, playing Fibonacci nim with several heaps, based on where the move dynamic lives: is the move dynamic global, or is it local, specific to each heap?

In this article, we consider the move dynamic to be local, so there is a separate move dynamic assigned to each heap, and a move in one heap does not change the move dynamic in any other heap. The reason for this is that this rule fits in properly with the Sprague-Grundy theory, as it is simply the disjunctive sum operator. This allows us to analyze the game assuming we can analyze each heap separately.

In order to analyze each heap, it is necessary to compute Grundy values of single-heap positions. We consider a position to be a pair $(n,r)$, where $n$ is the total number of stones in the heap, and $r$ is the maximum number that may be removed on the next turn. The starting position is therefore $(n,n-1)$. We sometimes simply write $n$ to denote $(n,n)$.

In Table~\ref{table:gvals}, we display the Grundy values of the pairs $(n,r)$ for small values of $n$ and $r$. We write $\mcG(n,r)$ for the Grundy value of the pair $(n,r)$.

\begin{table} \begin{tabular}{c||c|cccc|cccc|cccc|cccc|cccc} $n\backslash r$ & 0 & 1 & 2 & 3 & 4 & 5 & 6 & 7 & 8 & 9 & 10 & 11 & 12 & 13 & 14 & 15 & 16 & 17 & 18 & 19 & 20 \\ \hline 0 & 0 \\ 1 & 0 & 1 \\ 2 & 0 & 0 & 2 \\ 3 & 0 & 0 & 0 & 3 \\ 4 & 0 & 1 & 1 & 3 & 3 \\ 5 & 0 & 0 & 0 & 0 & 0 & 4 \\ 6 & 0 & 1 & 1 & 1 & 1 & 4 & 4 \\ 7 & 0 & 0 & 2 & 2 & 2 & 4 & 4 & 4 \\ 8 & 0 & 0 & 0 & 0 & 0 & 0 & 0 & 0 & 5 \\ 9 & 0 & 1 & 1 & 1 & 1 & 1 & 1 & 1 & 5 & 5 \\ 10 & 0 & 0 & 2 & 2 & 2 & 2 & 2 & 2 & 5 & 5 & 5 \\ 11 & 0 & 0 & 0 & 3 & 3 & 3 & 3 & 5 & 5 & 5 & 5 & 5 \\ 12 & 0 & 1 & 1 & 3 & 3 & 3 & 3 & 3 & 6 & 6 & 6 & 6 & 6  \\ 13 & 0 & 0 & 0 & 0 & 0 & 0 & 0 & 0 & 0 & 0 & 0 & 0 & 0 & 6 \\ 14 & 0 & 1 & 1 & 1 & 1 & 1 & 1 & 1 & 1 & 1 & 1 & 1 & 1 & 6 & 6 \\ 15 & 0 & 0 & 2 & 2 & 2 & 2 & 2 & 2 & 2 & 2 & 2 & 2 & 2 & 6 & 6 & 6 \\ 16 & 0 & 0 & 0 & 3 & 3 & 3 & 3 & 3 & 3 & 3 & 3 & 3 & 7 & 7 & 7 & 7 & 7 \\ 17 & 0 & 1 & 1 & 3 & 3 & 3 & 3 & 3 & 3 & 3 & 3 & 3 & 3 & 7 & 7 & 7 & 7 & 7 \\ 18 & 0 & 0 & 0 & 0 & 0 & 4 & 4 & 4 & 4 & 4 & 4 & 7 & 7 & 7 & 7 & 7 & 7 & 7 & 7 \\ 19 & 0 & 1 & 1 & 1 & 1 & 4 & 4 & 4 & 4 & 4 & 4 & 4 & 7 & 7 & 7 & 7 & 7 & 7 & 7 & 7 \\ 20 & 0 & 0 & 2 & 2 & 2 & 4 & 4 & 4 & 4 & 4 & 4 & 4 & 4 & 7 & 7 & 7 & 7 & 7 & 7 & 7 & 7  \end{tabular} \caption{Grundy values for Fibonacci nim} \label{table:gvals} \end{table}

The structure of the rest of the paper is as follows. In \S\ref{sec:sprague-grundy}, we review the Sprague-Grundy theory. In \S\ref{sec:playing}, we review Zeckendorf's theorem and the winning strategy for Fibonacci nim. In \S\ref{sec:smallvals}, we give a complete description of the positions $(n,r)$ with $\mcG(n,r)\le 3$. In \S\ref{sec:startingvals}, we show that the nonzero Grundy values of the starting positions are increasing and provide upper and lower bounds for their sizes.

\section*{Acknowledgements}

This project was started at the TRUe Games workshop at Thompson Rivers University in Kamloops, British Columbia, in May 2014. The conference was sponsored by the Pacific Institute for the Mathematical Sciences. The authors would like to thank the anonymous referees for their helpful suggestions.

\section{The Sprague-Grundy theory} \label{sec:sprague-grundy}

When analyzing an impartial two-player game in isolation, it is sufficient to identify the $\mcN$ positions, which are winning for the next player, and the $\mcP$ positions, which are winning for the previous player (or, equivalently, losing for the next player). These positions can be classified recursively, as follows:

\begin{itemize}

\item A position is an $\mcN$ position if there is at least one move to a $\mcP$ position.

\item A position is a $\mcP$ position if every move is to an $\mcN$ position.

\end{itemize}

It is possible to analyze a sum of several games by understanding each game individually, but it is necessary to know more detailed information than just whether it is an $\mcN$ or $\mcP$ position. The key is the minimal excludant (mex) function.

\begin{defn} Let $S$ denote a finite set of nonnegative integers. Then the minimal excludant $\mex(S)$ is the least nonnegative integer not in $S$. \end{defn}

The Sprague-Grundy theory assigns a nonnegative integer $\mcG(X)$, known as the Grundy value of $X$, to each finite impartial game $X$ recursively, by letting $\mcG(X)=\mex(\{\mcG(Y)\})$, where $Y$ runs over all the moves from $X$.

If $X$ decomposes as a sum of several games, say $X=X_1+\cdots+X_n$, then $\mcG(X)=\mcG(X_1)\oplus\cdots\oplus\mcG(X_n)$, where the operator $\oplus$ is ``add in binary without carrying.'' (See e.g.~\cite{BCG01} for more details.) An impartial game $X$ is a $\mcP$ position iff $\mcG(X)=0$.

\section{Playing Fibonacci nim} \label{sec:playing}

The essential ingredient to winning at Fibonacci nim is Zeckendorf's Theorem.

\begin{thm}[Zeckendorf, \cite{Zeckendorf72}] Every positive integer has a unique representation as a sum of distinct Fibonacci numbers, no two of which are consecutive. \end{thm}

We call this representation the Zeckendorf representation of $n$. We write $z_i(n)$ for the $i^\text{th}$ smallest part in the Zeckendorf representation of $n$; if the Zeckendorf representation of $n$ contains fewer than $i$ parts, then we write $z_i(n)=\infty$. We also write expressions of the form $n=a+b+c+\cdots$, meaning that $z_1(n)=a$, $z_2(n)=b$, $z_3(n)=c$, and $c<\infty$.

Now, assume that $(n,r)$ is an $\mcN$ position. As we shall prove in Theorem~\ref{thm:zeroes}, this is true if and only if $r\ge z_1(n)$. A winning move is to remove $z_1(n)$ tokens. (There may be other winning moves as well.)

Because the winning strategy of Fibonacci nim is so closely tied to Zeckendorf's Theorem, we can view the entire game as a game-theoretic interpretation of Zeckendorf's Theorem.

\section{Small Grundy Values} \label{sec:smallvals}

\begin{notn} We write $F_t$ for the $t^\text{th}$ Fibonacci number. As usual, we index the Fibonacci numbers so that $F_0=0$ and $F_1=1$. \end{notn}

We show the following:

\begin{thm} \label{thm:zeroes} $\mcG(n,r)=0$ if and only if $r<z_1(n)$. \end{thm}

\begin{rem} An important special case of Theorem~\ref{thm:zeroes} is that the starting position $(n,n-1)$ with $n$ stones is losing iff $n$ is a Fibonacci number. \end{rem}

This is a classical result, due to Whinihan in~\cite{Whinihan63}. However, its proof will be useful for the rest of our results, so we review it here. We will make use of the following Lemma:

\begin{lem} \label{lem:smallfibs} Suppose $n>1$ and $1\le k<z_1(n)$. If $z_1(k)=F_t$, then $z_1(n-k)$ is either $F_{t+1}$ or $F_{t-1}$. In particular, $z_1(n-k) \le 2k$, and if $k\ge 4$, then $z_1(n-k)\le 2k-2$. \end{lem}

\begin{rem} \label{rem:clauses} We primarily use the clause that $z_1(n-k)\le 2k$. However, at one point in the proof of Theorem~\ref{thm:threes}, we will need the stronger clause that $z_1(n-k)\le 2k-2$ if $k\ge 4$. \end{rem}

\begin{proof} We prove this by induction on the number of parts in the Zeckendorf representation of $k$. We start with the case of $k$ being a Fibonacci number, so that $z_1(k)=k$. Suppose that $z_1(n)=F_s$. We divide the proof into two cases: $s\equiv t\pmod 2$ and $s\not\equiv t\pmod 2$. If $s\equiv t\pmod 2$, then we have $t=s-2d$ for some $d\ge 1$, and we have \[F_s-k=F_s-F_{s-2d}=F_{s-2d+1}+F_{s-2d+3}+\cdots+F_{s-3}+F_{s-1},\] so \[z_1(F_s-k)=F_{s-2d+1}=F_{t+1}.\] Now, note that the Zeckendorf representation of $n-k$ is equal to the union of the Zeckendorf representation of $F_s-k$ and the Zeckendorf representation of $n$ with the first part (that is, $F_s$) removed. So, the result holds in this case.

Now, suppose that $s\not\equiv t\pmod 2$. Then $t=s-2d-1$ for some $d\ge 0$, and we have \[F_s-F_t=F_s-F_{s-2d-1}=F_{s-2d-2}+F_{s-2d}+\cdots+F_{s-3}+F_{s-1},\] so $z_1(F_s-k)=F_{s-2d-2}=F_{t-1}$. As before, we have $z_1(n-k)=z_1(F_s-k)$, so here too the result holds.

Now suppose that the result holds whenever the Zeckendorf representation of $k$ has $p-1$ parts. Suppose furthermore that the Zeckendorf representation of $k$ has $p$ parts. Then, since $k-z_1(k)$ has $p-1$ parts, we know that if $z_1(k)=F_t$, then $z_1(k-z_1(k))\ge F_{t+2}$, so $z_1(n-k+z_1(k))\ge F_{t+1}>F_t=z_1(k)$. Hence, by the base case above with $n-k+z_1(k)$ and $z_1(k)$, respectively, playing the parts of $n$ and $k$, $z_1(n-k)$ is either $F_{t-1}$ or $F_{t+1}$. \end{proof}

\begin{proof}[Proof of Theorem~\ref{thm:zeroes}] The proof of this theorem, and the others in this section, are all by induction on $n$. It suffices to show that, from any position with $r\ge z_1(n)$, there is some $k$ with $k\le r$ so that $2k<z_1(n-k)$ (in fact, $k=z_1(n)$ works), and that if $r<z_1(n)$, then for every $k\le r$, $2k\ge z_1(n-k)$. In the language of $\mcN$ and $\mcP$ positions, this says that for every $\mcN$ position ($r\ge z_1(n)$), there is a move to a $\mcP$ position ($r<z_1(n)$), and for every $\mcP$ position, all moves are to $\mcN$ positions.

Assume that $r\ge z_1(n)$. We show that $2z_1(n)<z_2(n)=z_1(n-z_1(n))$. Since $z_1(n)$ is a Fibonacci number, say $F_t$ with $t\ge 2$, and $z_2(n)$ is also a Fibonacci number at least $F_{t+2}$, we have \[z_2(n)\ge F_{t+2}=F_{t+1}+F_t>2F_t,\] as desired. Hence, $k=z_1(n)$ satisfies the condition in the previous paragraph.

Now assume that $k<z_1(n)$. By Lemma~\ref{lem:smallfibs}, if $z_1(k)=F_t$, then $z_1(n-k)\le F_{t+1}$. Since $F_{t+1}\le 2F_t\le 2k$, we have $2k\ge z_1(n-k)$, as desired. \end{proof}

\begin{thm} \label{thm:ones} $\mcG(n,r)=1$ iff $z_1(n)=1$ and $1\le r<z_2(n)$. \end{thm}

\begin{proof} In order for $\mcG(n,r)$ to be 1, there must be some $k$ with $1\le k\le r$ so that $\mcG(n-k,2k)=0$, and furthermore, $\mcG(n-k,2k)\neq 1$ for all $k$ with $1\le k\le r$. Suppose $z_1(n)=1$ and $r<z_2(n)$. Then $\mcG(n-1,2)=0$ by Theorem~\ref{thm:zeroes}, since $z_1(n-1)=z_2(n)\ge 3$, as otherwise the Zeckendorf representation of $n$ would have two consecutive Fibonacci numbers, which is impossible. We now show that, for each $k<z_2(n)=z_1(n-1)$, $\mcG(n-k,2k)\neq 1$. It suffices to show that either $z_1(n-k)\neq 1$ or $2k\ge z_2(n-k)$. This follows from applying Lemma~\ref{lem:smallfibs} with $n-1$ in place of $n$, since if $z_1(n-k)=1$, then $z_2(n-k)=z_1(n-1-k)$.

Now, suppose that $z_1(n)>1$. If $r<z_1(n)$, then by Theorem~\ref{thm:zeroes}, $\mcG(n,r)=0$. If $r\ge z_1(n)$, then there is a move to $(n-z_1(n)+1,2z_1(n)-2)$. Now, $z_1(n-z_1(n)+1)=1$, and $2z_1(n)-2\le z_2(n-z_1(n)+1)=z_2(n)$. Hence $\mcG(n-z_1(n)+1,2z_1(n)-2)=1$, so $\mcG(n,r)\neq 1$.

Finally, suppose that $z_1(n)=1$ and $r\ge z_2(n)$. Then $\mcG(n-z_2(n),2z_2(n))=1$, since $2z_2(n)<z_3(n)=z_2(n-z_2(n))$. Thus, in this case, there is a move to a position with Grundy value 1, so $\mcG(n,r)\neq 1$. \end{proof}

\begin{thm} \label{thm:twos} $\mcG(n,r)=2$ iff $z_1(n)=2$ and $2\le r<z_2(n)$. \end{thm}

\begin{proof} In order for $\mcG(n,r)$ to be 2, there must be moves to positions of values 0 and 1, and no move to a position of value 2. We now show that if $z_1(n)=2$ and $2\le r<z_2(n)$, then $\mcG(n,r)=2$. Since $r\ge 2$ and $z_1(n)=2$, $\mcG(n-2,4)=0$, so there is a move to a 0-position, since $z_1(n-2)=z_2(n)\ge 5$. Furthermore, $\mcG(n-1,2)=1$, so there is a move to a 1-position. Now, suppose that $\mcG(n-k,2k)=2$ for some $k\le r$. Then, by induction, we would have $z_1(n-k)=2$ and $2k<z_2(n-k)$. But if $z_1(n-k)=2$, then $z_2(n-k)=z_1(n-k-2)$, which, since $k\le r<z_2(n)=z_1(n-2)$, is $\le 2k$ by Lemma~\ref{lem:smallfibs}, which contradicts the induction. Hence, there are no moves to positions of value 2. 

Now, suppose $z_1(n)\neq 2$. If $z_1(n)=1$ and $1\le r<z_2(n)$, then by Theorem~\ref{thm:ones}, $\mcG(n,r)=1$. Now, suppose $z_1(n)=1$ and $r\ge z_2(n)$. Then $\mcG(n-z_2(n)+1,2z_2(n)-2)=2$ by induction. Hence, in this case, $\mcG(n,r)\neq 2$.

Finally, suppose $z_1(n)=2$ and $r\ge z_2(n)$. Then $\mcG(n-z_2(n),2z_2(n))=2$, since $2z_2(n)<z_3(n)$. Hence, there is a move to a position with Grundy value 2, so $\mcG(n,r)\neq 2$. \end{proof}

\begin{thm} \label{thm:threes} $\mcG(n,r)=3$ iff $z_1(n)=1$, $z_2(n)=3$, and $3\le r<z_3(n)$, or $z_1(n)=3$ and $3\le r<z_2(n)-1$. \end{thm}

\begin{proof} In order for $\mcG(n,r)$ to be 3, there must be moves to positions of values 0, 1, and 2, and no move to a position of value 3. We now show that if $z_1(n)=1$, $z_2(n)=3$, and $3\le r<z_3(n)$, then $\mcG(n,r)=3$. By Theorem~\ref{thm:zeroes}, $\mcG(n-1,2)=0$, so there is a move to 0. By Theorem~\ref{thm:ones}, $\mcG(n-3,6)=1$, since $n-3=1+z_3(n)+\cdots$ and $z_3(n)\ge 8$ by Zeckendorf's Theorem, since $z_2(n)=3$. By Theorem~\ref{thm:twos}, $\mcG(n-2,4)=2$, since $n-2=2+z_3(n)+\cdots$. Now, we show that there are no moves from $(n,r)$ to a position with Grundy value 3. Clearly, removing one or two tokens does not leave a position with Grundy value 3. If we were to leave a position with Grundy value 3 after removing $3\le k<z_3(n)$, then we must either have $z_1(n-k)=1$, $z_2(n-k)=3$, and $2k<z_3(n-k)$, or $z_1(n-k)=3$ and $2k<z_2(n-k)-1$. In the first case, we have $z_3(n-k)=z_1(n-k-4)$, and as $k<z_3(n)=z_1(n-4)$, Lemma~\ref{lem:smallfibs} implies that $z_3(n-k)\le 2k$, contradicting the hypothesis. In the second case, we have $z_2(n-k)=z_1(n-k-3)\le 2(k-1)$ by Lemma~\ref{lem:smallfibs}, contradicting the assumption that $z_2(n-k)>2k+1$. Hence, there is no move to a position with Grundy value 3.

Now suppose that $z_1(n)=3$ and $3\le r<z_2(n)-1$. Then $\mcG(n-3,6)=0$ since $6<z_1(n-3)=z_2(n)$. Now, $\mcG(n-2,4)=1$ since $n-2=1+z_2(n)+\cdots$ and $z_2(n)\ge 8$. Furthermore, $\mcG(n-1,2)=2$ since $n-1=2+z_2(n)+\cdots$. If there were a move to a position $(n-k,2k)$ of Grundy value 3, then we would either have $z_1(n-k)=1$, $z_2(n-k)=3$, and $2k<z_3(n-k)$, or $z_1(n-k)=3$ and $2k<z_2(n-k)-1$. Furthermore, if $k\le 3$, then we have already seen that $\mcG(n-k,2k)\neq 3$, so we may assume that $k\ge 4$, putting us in the final case of Lemma~\ref{lem:smallfibs}, as mentioned in Remark~\ref{rem:clauses}. In the first case, $z_3(n-k)=z_1(n-k-4)$, and as $k+1<z_2(n)=z_1(n-3)$, Lemma~\ref{lem:smallfibs} implies that $z_3(n-k)\le 2(k+1)-2=2k$, so by induction $\mcG(n-k,2k)\neq 3$. In the second case, $z_2(n-k)=z_1(n-k-3)$, and by Lemma~\ref{lem:smallfibs}, $z_1(n-k-3)\le 2k$, contradicting the hypothesis. Hence, once again there is no move to a position with Grundy value 3.

Now, we must show that for any $(n,r)$ not of the above two forms, $\mcG(n,r)\neq 3$. If $n=1+3+z_3(n)+\cdots$ and $r<3$, then there are only at most two moves, so there are only at most two Grundy values among its moves, so $\mcG(n,r)<3$. This is also true if $n=3+z_2(n)+\cdots$ and $r<3$. Now, if $n=1+3+z_3(n)+\cdots$ and $r\ge z_3(n)$, then we can remove $z_3(n)$ tokens to obtain $(n-z_3(n),2z_3(n))$, which has Grundy value 3 by induction. Similarly, if $n=3+z_2(n)+\cdots$ and $r\ge z_2(n)-1$, then we can remove $z_2(n)-1$ tokens to reach $(n-z_2(n)+1,2z_2(n)-2)$, which has value 3 by induction. Hence, these positions do not have Grundy value 3.

Now, suppose $n=1+z_2+\cdots$, where $z_2\ge 5$. If $r<z_2$, then $\mcG(n,r)\le 1$. If $r\ge z_2$, then there is a move to $(n-z_2+2,2z_2-4)$, which has Grundy value 3 by induction, so $\mcG(n,r)\neq 3$. Now suppose $n=2+z_2+\cdots$. If $r<z_2$, then $\mcG(n,r)\le 2$. If $r\ge z_2$, then there is a move to $(n-z_2+1,2z_2-2)$, which has Grundy value 3 by induction, so $\mcG(n,r)\neq 3$. Finally, suppose $z_1(n)\ge 5$. If $r<z_1(n)$, then $\mcG(n,r)=0$. If $r\ge z_1(n)$, then there is a move to $(n-z_1(n)+3,2z_1(n)-3)$, which has Grundy value 3. Hence $\mcG(n,r)\neq 3$. This completes the proof. \end{proof}

It appears to be more difficult to classify the positions of Grundy value $k$ for $k\ge 4$. Thus we turn to the problem of understanding the Grundy values of the initial positions $(n,n-1)$ (and $(n,n)$) and their growth.

\section{Values of starting positions} \label{sec:startingvals}

In this section we prove the following result.

\begin{thm}\label{thm:4} 
Ignoring the Fibonacci numbers, the Grundy values $\mcG(n,n-1)$ of the starting positions are non-decreasing. Furthermore, when they increase, they increase by one. 
\end{thm}

Consider positions of the form $(n,n)$. A starting position is of the form $(n, n-1)$. Unless $n$ is a Fibonacci number, it is clear that $\mcG(n,n) = \mcG(n,n-1)$, since the only additional move is to $(0,0)$, which has Grundy value 0. Recall that we sometimes denote a position of the form $(n,n)$ simply by $n$. It is clear that Theorem~\ref{thm:4} follows from the theorem below, which we prove instead.

\begin{thm}\label{thm:increasing} For all $n\ge 0$, $\mcG(n)\le\mcG(n+1)\le\mcG(n)+1$. \end{thm}

Before we begin the proof, we introduce some notation. For each $g\ge 0$, let $h(g)$ be the smallest value of $n$ for which there is some $r$ with $\mcG(n,r)=g$. It is clear that we could equivalently let $h(g)$ be the smallest value of $n$ for which $\mcG(n,n)=g$. For $g\ge 0$, let $A_g$ denote the set of pairs $(n,r)$ with $n<h(g+1)$ and for which $\mcG(n,r)=g$. We think of $A_g$ as being the ``first block'' of positions $(n,r)$ for which $\mcG(n,r)=g$. A key property of $A_g$ is that if $(n,r)\in A_g$ and $r'>r$, then $(n,r')\in A_g$ as well.

\begin{proof}
We prove the theorem by induction on $n$, together with the following statement: if $\mcG(n)=g$, then for each $d<g$, there is some move from $n$ to $(m_d,k_d)$ with $(m_d,k_d)\in A_d$. For $n=0$, both of these statements are clear. Now, suppose they both hold for $n$; we show that they also hold for $n+1$. Suppose $\mcG(n)=g$. Then, for each $d<g$, there is a move from $n$ to $(m_d,k_d)\in A_d$. Thus there is a move from $n+1$ to $(m_d,k_d+2)$. Since $(m_d,k_d+2)\in A_d$, we have a move from $n+1$ to a position in $A_d$. Hence, there are moves from $n+1$ to positions of Grundy value $d$ for all $d<g$, so $\mcG(n+1)\ge g$. The inductive hypothesis shows that $h(g+1)\ge n+1$, so any position $(m,k)$ with $m<n+1$ and $\mcG(m,k)=g$ must be in $A_g$; furthermore, $h(g+2)>n+1$. Thus, if there is a move from $n+1$ to a position $(m,k)$ with Grundy value $g$, then $(m,k)\in A_g$. This completes the proof.
\end{proof}

\begin{thm}\label{thm:4.1} We have $\log_{3/2}(n)\le\mcG(n)\le \lceil2\sqrt{n}\rceil+1$.
\end{thm}

\begin{proof} We first prove the lower bound. For $n>0$, let $n'=\lceil\frac{3n}{2}\rceil$. We show that $\mcG(n')\ge\mcG(n)+1$, which implies the lower bound $\log_{3/2}(n)\le\mcG(n)$. By Theorem~\ref{thm:increasing}, we have $\mcG(r)\le\mcG(r+1)\le\mcG(r)+1$ for all $r$. From $n'$, there is a move to $n$, and hence to $r$ for each $r\le n$. Thus, the moves from $n'$ include moves to $0,1,2,\ldots,n$, and $\{\mcG(0),\mcG(1),\ldots,\mcG(n)\}=\{0,1,\ldots,\mcG(n)\}$. Hence $\mcG(n')\ge\mcG(n)+1$.

To prove the upper bound, we let $j(g)$ be the least value of $r$ for which there is some $n$ with $\mcG(n,r)=g$. In order for $\mcG(n,r)$ to be equal to $g$, there must be at least $g$ moves from $(n,r)$, since there must be moves to positions of value $0,1,2,\ldots,g-1$. Hence, $j(g)\ge g$. Now, assuming we have computed $h(g)$, we give a lower bound for $h(g+1)$. In order for $\mcG(n,r)$ to be equal to $g+1$, there must be a move to a position $(n_1,r_1)$ whose Grundy value is $g$. Hence, we need $n_1\ge h(g)$ and $r_1\ge j(g)\ge g$. Since $r_1=2(n-n_1)$, we obtain $2(n-n_1)\ge g$, or $n\ge \frac{g}{2}+n_1\ge \frac{g}{2}+h(g)$, so $h(g+1)-h(g)\ge\frac{g}{2}$. Since $h(1)=1$, we have \[h(g)-1=\sum_{i=1}^{g-1} (h(i+1)-h(i))\ge \sum_{i=1}^{g-1} \frac{i}{2}=\frac{g(g-1)}{4},\] so $h(g)\ge\frac{g(g-1)}{4}$. Thus, \[h(\lceil 2\sqrt{n}\rceil+1)\ge\frac{(2\sqrt{n}+1)2\sqrt{n}}{4}+1>n,\] so $\mcG(n)\le\lceil 2\sqrt{n}\rceil+1$, as desired. \end{proof}

In fact, it appears that the lower bound is a lot closer to the truth than is the upper bound. More precisely, we conjecture based on numerical evidence that $\mcG(n)+1\le \mcG(\lceil\frac{3n}{2}\rceil) \le \mcG(n)+2$, which would imply that the growth rate is logarithmic.

\bibliographystyle{alpha}
\bibliography{fibnim}

\end{document}